\documentclass[11pt]{amsart}

\usepackage{latexsym}
\usepackage{amssymb}
\usepackage{amsmath}
\usepackage{enumerate}
\usepackage{color}
\usepackage{comment}
\usepackage{a4wide}

\newtheorem{theorem}{Theorem}[section]
\newtheorem{lemma}[theorem]{Lemma}
\newtheorem{proposition}[theorem]{Proposition}
\newtheorem{corollary}[theorem]{Corollary}
\newtheorem{definition}[theorem]{Definition}

\newtheorem{remark}[theorem]{Remark}
\newtheorem{remarks}[theorem]{Remarks}

\newcommand{\cl}[1]{\mathcal{#1}}
\newcommand{\bb}[1]{\mathbb{#1}}

\newcommand{\NN}{\ensuremath{\mathbb{N}}} 
\newcommand{\Z}{\ensuremath{\mathbb{Z}}} 
\newcommand{\CC}{\ensuremath{\mathbb{C}}} 

\newcommand\nph{\varphi} 
\def\inv{^{-1}} 
\newcommand{\unA}{1_A} 
\newcommand{\ip}[2]{\ensuremath{\left\langle #1 , #2\right\rangle}} 
\newcommand\supp{\mathop{\rm supp}} 
\newcommand\id{\mathop{\rm id}} 
\newcommand\cb{\mathop{\rm cb}} 
\newcommand\pii[1]{\pi(#1)} 

\begin{document}

\title[Positive Herz-Schur multipliers and approximations]{Positive Herz-Schur multipliers and approximation properties of crossed products}

\author[A. McKee]{A. McKee}
\address{Mathematical Sciences Research Centre, Queen's University Belfast, Belfast BT7 1NN, United Kingdom}
\email{amckee240@qub.ac.uk}

\author[A. Skalski]{A. Skalski}
\address{Institute of Mathematics of the Polish Academy of Sciences,
ul.~\'Sniadeckich 8, 00--656 Warszawa, Poland}
\email{a.skalski@impan.pl}

\author[I. G. Todorov]{I. G. Todorov}
\address{Mathematical Sciences Research Centre, Queen's University Belfast, Belfast BT7 1NN, United Kingdom}
\email{i.todorov@qub.ac.uk}

\author[L. Turowska]{L. Turowska}
\address{Department of Mathematical Sciences,
Chalmers University of Technology and  the University of Gothenburg,
Gothenburg SE-412 96, Sweden}
\email{turowska@chalmers.se}


\begin{abstract}
For a $C^*$-algebra $A$ and a set $X$ we give a Stinespring-type characterisation of the  completely positive Schur $A$-multipliers on $\cl K(\ell^2(X))\otimes A$.
We then relate them to completely positive Herz-Schur multipliers on $C^*$-algebraic crossed products of the form $A\rtimes_{\alpha,r} G$, with $G$ a discrete group, whose various versions were considered earlier by Anantharaman-Delaroche, B\'edos and Conti, and Dong and Ruan.
The latter maps are shown to implement approximation properties, such as nuclearity or the Haagerup property, for $A\rtimes_{\alpha,r} G$.
\end{abstract}

\subjclass[2010]{Primary: 46L55, Secondary: 43A35, 46L05}

\keywords{completely positive Schur $A$-multipliers; C*-crossed products; approximation properties}

\maketitle

\section{Introduction}

\noindent
Schur multipliers, a class of maps generalising the operators of entrywise (Schur) multiplication on finite matrices, were first abstractly studied by Grothendieck in his famous `R\'esum\'e' \cite{Gro}. 
Since then they have played a very important role in operator theory, the theory of Banach spaces, and operator space theory (for a longer discussion we refer to the introduction of \cite{mtt}, and to Section 5 of \cite{Pisier}).
In the simplest situation they arise in the following manner: to a (discrete) set $X$ and a function $\phi: X \times X \to \mathbb{C}$,
one associates an operator $S_\phi$ on the space of compact operators on the Hilbert space $\ell^2(X)$; if the resulting map is (completely) bounded, we call $S_\phi$ a 
Schur multiplier with symbol $\phi$.
A related class of maps is that of the Herz-Schur multipliers associated to a discrete (or, more generally, locally compact) group $G$.
On one hand, these can be viewed as the special class of Schur multipliers
whose symbol $\phi : G \times G \to \mathbb{C}$ is `invariant'.
On the other hand, they form a natural extension of the Fourier multipliers from
classical harmonic analysis; in particular, in this case,
the operator $S_\phi$ can be viewed as acting on the reduced group $C^*$-algebra $C^*_r(G)$ of $G$.

Since the pioneering work of Haagerup, 
it has been known that the existence of Herz-Schur multipliers of a particular type encodes various approximation properties of $C^*_r(G)$, see Chapter 12 of \cite{bo} (it is worth mentioning that the same fact persists for discrete unimodular \emph{quantum groups} --- see the recent survey \cite{Brannan}).
An important ingredient of the proof of such results, which features again in this paper, can be summarised as follows: if $(\phi_i)_{i \in I}$ is a net of Herz-Schur multipliers on $G$ with certain properties then the associated operators $S_{\phi_i}$ implement an approximation property of $C^*_r(G)$;  conversely any family of approximating maps on $C^*_r(G)$ can be `averaged' into Herz-Schur multipliers. An early example of this technique is Lance's proof \cite{lance} that a discrete group is amenable if and only if its reduced group $C^*$-algebra is nuclear.

Recently both classes of maps discussed above have been generalised to the $C^*$-algebra-valued case. 
In the paper \cite{mtt}, written by three authors of the current article, a class of \emph{Schur $A$-multipliers} is identified, where $A$ is a $C^*$-algebra faithfully represented on a Hilbert space $H$.
In this `operator-valued' case,
the starting point is a function $\phi$, defined on the direct product $X \times X$, and taking values in
the space $CB(A, \cl B(H))$ of all completely bounded maps from $A$ into
the $C^*$-algebra $\cl B(H)$ of all bounded linear operators on $H$.
The associated operator $S_\phi$ acts from
$\cl K(\ell^2(X))\otimes A$ to $\cl K(\ell^2(X))\otimes B(H)$; the function $\phi$ is called a Schur $A$-multiplier if the map $S_\phi$ is completely bounded.
Again, in the case where $X$ is in fact a discrete group $G$, the $C^*$-algebra $A$ is equipped with an action of $G$, and the function $\phi$ satisfies a natural invariance property with respect to the action, we are led to a generalisation of Herz-Schur multipliers, this time acting on the $C^*$-algebraic crossed product $A\rtimes_{\alpha,r} G$.
Such maps have also been studied in a series of papers by B\'edos and Conti (see \cite{bc2} and references therein), where they are viewed as generalisations of Fourier multipliers.
In \cite{mtt} several general properties and examples of
Schur $A$-multipliers were established; later, in \cite{Andrew}, it was shown that the completely bounded approximation property of reduced crossed products can be characterised {\it via} the existence of Herz-Schur $A$-multipliers of a particular type.

The descriptions above imply that the operator-space-theoretic concept of complete boundedness plays an important role in the theory of Schur and Herz-Schur multipliers.
In the study of approximation properties of $C^*$-algebras it is well known that a special role is played by \emph{completely positive} maps. These are the subject of our paper.
We first characterise, in Theorem \ref{re:Sphicpintermsofphi},
those functions $\phi:X \times X \to CB(A, \cl B(H))$ for which $S_\phi$ is a completely positive Schur $A$-multiplier, calling such
$\phi$ \emph{functions of positive type}.
Then we prove Theorem~\ref{co:equivalencesofpositiveHSmultiplier}, a transference-type result: in the case where
$X$ is a discrete group $G$ acting on $A$ by automorphisms, there is a
one-to-one correspondence between certain `invariant' functions $\phi:G \times G \to CB(A)$ of positive type, and functions $F:G \to CB(A)$ leading to completely positive Herz-Schur multipliers on the reduced crossed product $A\rtimes_{\alpha,r} G$.
The latter class of maps is then compared to those studied earlier by Anantharaman-Delaroche~\cite{Claire}, B\'edos and Conti~\cite{bc2}, and Dong and Ruan~\cite{dong_ruan}.
These general results are later applied to show that approximation properties of $C^*$-algebraic crossed products can be realised using Herz-Schur $A$-multipliers, generalising the corresponding statements for reduced group $C^*$-algebras (which can be viewed as `crossed products with trivial coefficients'), mentioned earlier in this introduction.

The plan of the paper is as follows: 
 after finishing this section by introducing some general notational conventions, in Section 2 we discuss functions of positive type and associated completely positive Schur $A$-multipliers, establishing a Stinespring-type representation for the latter.
We also introduce the related completely positive Herz-Schur $A$-multipliers, which act on reduced crossed products, and compare the resulting class with those considered earlier by other authors.
Section~3 presents a characterisation of the Haagerup property for the reduced crossed product in terms of the existence of a certain class of completely positive Herz-Schur $A$-multipliers; Section~4 solves the analogous problem for nuclearity.

For $C^*$-algebras $A$ and $B$, we denote by $A \otimes B$ their minimal/spatial tensor product;
$\mathcal{Z}(A)$ stands for the centre of $A$, while $A^+$ designates the cone of positive elements of $A$.
Scalar products will be always linear on the left.
For Hilbert spaces $H$ and $K$, we denote by $H\otimes K$ their Hilbertian tensor product.
We let $\cl B(H,K)$ be the space of all bounded linear operators from $H$ into $K$ and set $\cl B(H) = \cl B(H,H)$.
For a vector space $V$, we denote, as usual, by $M_n(V)$ the space of all $n$ by $n$ matrices with entries in $V$.
For a map $\phi : V\to W$ between vector spaces $V$ and $W$, we let $\phi^{(n)} : M_n(V)\to M_n(W)$
be the map given by $\phi^{(n)}((x_{i,j})_{i,j}) = (\phi(x_{i,j}))_{i,j}$.
For operator spaces $\cl X$ and $\cl Y$, we denote by $CB(\cl X,\cl Y)$ the (operator) space of all
completely bounded maps from $\cl X$ into $\cl Y$.
The paper will rely on acquaintance with basic operator space theory; we refer the reader to \cite{er_book} and \cite{paulsen}
for necessary background.


\section{Completely positive Herz-Schur multipliers}\label{s_psam}


\noindent
We recall some notions and results from \cite{mtt}.
Let $X$ be a set, $H$ a Hilbert space, and $A\subseteq\cl \cl B(H)$ a (nondegenerate) $C^*$-algebra.
We write $\cl K = \cl K(\ell^2(X))$ for the algebra of all compact operators on $\ell^2(X)$.
Set $\cl H = \ell^2(X)\otimes H$; we identify $\cl H$ with the Hilbert space $\ell^2(X,H)$
of all square summable sequences in $H$.
It was noted in \cite{mtt} that, for $k\in \ell^2(X\times X,A)$, the formula
\[
    (T_k\xi)(x) = \sum_{y\in X}k(x,y) \big( \xi(y) \big), \quad x\in X,\ \xi\in\cl H,
\]
defines a bounded operator on $\cl H$ with $\|T_k\|\leq \|k\|_2$.
The functions $k : X\times X\to A$ will often be referred to as \emph{kernels}.
We let
\[
    \cl S_2(X\times X,A) = \{T_k : k\in \ell^2(X\times X,A)\}
\]
and note that
$\cl S_2(X\times X,A)$ is a dense subspace of the minimal tensor product $\cl K \otimes A$.
Given a bounded function $\nph : X\times X\to CB(A,\cl B(H))$ and $k\in \ell^2(X\times X,A)$,
we write $\nph\cdot k$ for the function in $\ell^2(X\times X,A)$
given by $(\nph\cdot k)(x,y) = \nph(x,y)(k(x,y))$, $x,y\in X$.
Note we are using here a different convention to that of \cite{mtt};
see also the comment before Definition \ref{de:positiveHSmult}.
We let
$S_{\nph} : \cl S_2(X\times X,A)\to \cl S_2(X\times X, \cl B(H))$ be the map given by
$S_{\nph}(T_k) = T_{\nph\cdot k}$; it is clear that $S_{\nph}$ is linear and bounded with respect to the norm $\|\cdot\|_2$.
If $S_{\nph}$ is completely bounded, when
$\cl S_2(X\times X,A)$ is equipped with the operator space structure arising from its inclusion into $\cl K \otimes A$,
the function $\nph$ is called a \emph{Schur $A$-multiplier}.
We denote the algebra of all Schur $A$-multipliers on $X\times X$
by $\frak{S}(X,A)$. For $\nph\in \frak{S}(X,A)$, we write $\|\nph\|_{\rm m} = \|S_{\nph}\|_{\rm cb}$.

We note that if in the above paragraph the map
$\nph(x,y) : A\to \cl B(H)$, for $x,y\in X$,
is only assumed to be bounded (as opposed to completely bounded), then the complete boundedness of the map
$S_{\nph} : \cl S_2(X\times X,A) \to \cl S_2(X\times X,A)$ implies that $\nph(x,y)$ is
in fact completely bounded.

Let $\nph \in \frak{S}(X,A)$. Then the map $S_{\nph}$ has a (unique) completely bounded extension to a map on $\cl K\otimes A$, which we denote in the same way.
The following result, with the additional assumptions that the set $X$ be countable and the $C^*$-algebra $A$ be separable, was established in \cite{mtt} as a special case of Theorem~2.6 therein.
An inspection of the proof shows that these assumptions are not needed when $X$ is a discrete set endowed with counting measure.

\begin{theorem}\label{th_mtt}
Let $X$ be a set, $H$ a Hilbert space, $A\subseteq \cl B(H)$ a $C^*$-algebra, and
$\nph : X\times X\to CB(A,\cl B(H))$ a bounded function.
The following are equivalent:
\begin{enumerate}[i.]
\item[(i)] $\nph$ is a Schur $A$-multiplier;
\item[(ii)] there exist a Hilbert space $K$, a non-degenerate $*$-representation $\rho : A\to \cl B(K)$,
and bounded functions $V,W : X\to \cl B(H,K)$ such that
\[
    \nph(x,y)(a) = V(x)^* \rho(a) W(y), \quad x,y\in X, a\in A.
\]
\end{enumerate}
Moreover, if (i) holds true then the functions $V$ and $W$ in (ii) can be chosen so that
$\|\nph\|_{\rm m} = \sup_{x\in X}\|V(x)\| \sup_{y\in X}\|W(y)\|$.
\end{theorem}

Suppose that $\nph\in \frak{S}(X,A)$.
The map $S_{\nph} : \cl K \otimes A \to \cl K \otimes \cl B(H)$
has a canonical extension to a map from $\cl B(\ell^2(X))\otimes A$ into the weak* spatial tensor product
$\cl B(\ell^2(X))\bar\otimes \cl B(H)$, which we now describe.
Let $\rho$, $V$ and $W$ be as in Theorem \ref{th_mtt} (ii).
Fix $T\in \cl B(\ell^2(X))\otimes A$ and write $T = (a_{x,y})_{x,y}$, where $a_{x,y}\in A$, $x,y\in X$.
Since $\rho$ is completely bounded, we have that $\rho^{(\infty)}(T)\stackrel{\text{def}}{=} (\rho(a_{x,y}))_{x,y}$
defines a bounded operator on $\ell^2(X)\otimes H$. Letting $V$ (resp.\ $W$) be the diagonal
operator from $\ell^2(X)\otimes H$ into $\ell^2(X)\otimes K$ with entries $V(x)$ (resp.\ $W(y)$),
we set
\begin{equation} \label{Phiform}
    \Phi(T) = V^*\rho^{(\infty)}(T)W.
\end{equation}
We have that $\Phi(T) = (V(x)^*\rho(a_{x,y})W(y))_{x,y}$ and thus
$\Phi$ agrees with the map $S_{\nph}$ on $\cl K\otimes A$. In the sequel, if needed we will use the same symbol, $S_{\nph}$, to denote the mapping $\Phi$. 
Note that the mapping $\Phi$ coincides with the restriction to $\cl B(\ell^2(X))\otimes A$ of the map
$\cl E\circ S_{\nph}^{**}$, where
$S_{\nph}^{**} : \cl B(\ell^2(X))\bar\otimes A^{**}\to \cl B(\ell^2(X))\bar\otimes \cl B(H)^{**}$
is the second dual of $S_{\nph}$ and
$\cl E : \cl B(\ell^2(X))\bar\otimes \cl B(H)^{**} \to \cl B(\ell^2(X))\bar\otimes \cl B(H)$
is the canonical projection, whose existence follows from the fact that $\cl B(H)$ is a dual Banach space.

Recall that a kernel $k: X\times X \to A$ is called \emph{hermitian} if $k(x,y)^* = k(y,x)$ for all $x,y \in X$, and that
$k$ is called \emph{positive definite} if
\[
    \big(k(x_i,x_j)\big)_{i,j} \in M_n(A)^+, \quad \mbox{ for all } x_1,\ldots,x_n\in X \mbox{ and all } n\in \bb{N} .
\]

\begin{proposition}\label{pr:kandTkpositivetogether}
Let $k\in \ell^2(X\times X,A)$ be a hermitian kernel. The operator $T_k$ is positive
if and only if $k$ is positive definite.
\end{proposition}
\begin{proof}
Suppose that $k$ is positive definite and $\xi \in \ell^2(X,H)$ is
finitely supported, say $\supp(\xi) = \{x_1,\dots,x_n\}$.
Then
\[
    \ip{T_k \xi}{\xi} = \sum_{x \in X}\ip{\sum_{y \in X}k(x,y)\xi(y)}{\xi(x)}
    = \sum_{i,j=1}^n\ip{k(x_i,x_j)\xi(x_j)}{\xi(x_i)} \geq 0.
\]
Since the $H$-valued functions with finite support are dense in $\ell^2(X,H)$, we conclude that the operator $T_k$ is positive.

Now suppose that $T_k$ is positive.
Let $\xi_1,\ldots,\xi_n\in H$ and $x_1,\ldots,x_n\in X$.
Define $\xi \in \ell^2(X,H)$ by letting
\[
    \xi(x) =   \begin{cases}
                    \xi_i   & x = x_i ; \\
                    0       & \text{otherwise.}
                \end{cases}
\]
Then
\[
    \sum_{i,j}\ip{k(x_i,x_j)\xi_j}{\xi_i} = \ip{T_k \xi}{\xi} \geq 0,
\]
so $k$ is positive definite.
\end{proof}

\begin{definition}\label{def:phipositive}
A bounded function $\varphi: X\times X\to CB(A,\cl B(H))$ will be called \emph{of positive type} if,
for any $x_1,\ldots,x_m\in X$ and any $a_{p,q}\in A$, $p,q = 1,\dots,m$, such that
$(a_{p,q})_{p,q}\in M_{m}(A)^+$, we have that
$$\big(\varphi(x_p,x_q)(a_{p,q})\big)_{p,q} \in M_{m}(\cl B(H))^+.$$
\end{definition}

\begin{remark}\label{cbautomatic}
{\rm If we only assume that each $\varphi(x,y)$ is bounded then being of positive type implies the complete boundedness of $\varphi(x,y)$ for all $x,y\in X$.
The latter statement follows by arguments similar to those in the proof of Theorem~\ref{re:Sphicpintermsofphi} below.}
\end{remark}

\begin{lemma}\label{r_pt}
Suppose that a bounded function $\varphi: X\times X\to CB(A,\cl B(H))$ is of positive type.
\begin{enumerate}[(i)]
    \item If $x_1,\ldots,x_m\in X$ and $C_{p,q}\in M_n(A)$, $p,q = 1,\dots,m$, are such that
$(C_{p,q})_{p,q}\in M_{mn}(A)^+$, then
\begin{equation}\label{eq_pt}
    \big(\varphi(x_p,x_q)^{(n)}(C_{p,q})\big)_{p,q} \in M_{mn}(\cl B(H))^+;
\end{equation}
    \item If $k\in \ell^2(X\times X,A)$ is a positive definite kernel then $\nph\cdot k$ is such, too.
\end{enumerate}
\end{lemma}
\begin{proof}
(i) Letting $C_{p,q} = (a^{p,q}_{i,j})_{i,j =1}^n$, we have that $(a_{i,j}^{p,q})_{p,i,q,j}\in M_{mn}(A)^+$.
Set $y_{p,i} = x_p$, $p = 1,\dots,m$, $i = 1,\dots,n$. Then
$$\left(\varphi(x_p,x_q)^{(n)}(C_{p,q})\right)_{p,q} = \left(\varphi(y_{p,i},y_{q,j})(a^{p,q}_{i,j})\right)_{p,i,q,j} \in M_{mn}(A)^+.$$

(ii) is straightforward from the definitions.
\end{proof}

The following theorem provides a characterisation of those Schur $A$-multi-pliers $\varphi$ for which $S_\varphi$ is completely positive.

\begin{theorem}\label{re:Sphicpintermsofphi}
Let $X$ be a set, $H$ a Hilbert space, $A\subseteq \cl B(H)$ a non-degenerate $C^*$-algebra, and
$\varphi : X\times X\to CB(A,\cl B(H))$ a bounded function.
The following are equivalent:
\begin{enumerate}
    \item[(i)] $\nph$ is of positive type;
    \item[(ii)] $\varphi$ is a Schur $A$-multiplier, $S_\varphi$ is completely positive as a map from $\cl K\otimes A$ to $\cl K\otimes\cl B(H)$, and $\|S_{\nph}\|_{\cb} = \sup_{x\in X}\|\nph(x,x)\|_{\cb}$;
    \item[(iii)] there exists a Hilbert space $K$, a non-degenerate $*$-representation $\rho : A\to \cl B(K)$,
and a bounded function $V : X\to\cl B(H,K)$, such that
    \[
        \varphi(x,y)(a)=V(x)^* \rho(a) V(y),\quad x,y\in X,\ a\in A .
    \]
\end{enumerate}
Further if these conditions hold then the canonical extension of $S_\varphi$ to a map on $\cl B(\ell^2(X))\bar\otimes\cl B(H)$ is also completely positive.
\end{theorem}
\begin{proof}
(i)$\Rightarrow$(ii)
Let $\nph : X\times X \to CB(A,\cl B(H))$ be of positive type and
assume that $k_{i,j} \in \ell^2(X\times X,A)$, $i,j = 1\dots,n$, are such that $T = (T_{k_{i,j}})_{i,j}$ is a positive operator.
Let $[n] = \{1,\dots,n\}$, $Y = [n]\times X$ and $k : Y\times Y\to A$ be given by
$k(i,x,j,y) = k_{i,j}(x,y)$.
After identifying $\ell^2(Y,H)$ with $\bb{C}^n\otimes \ell^2(X,H)$, and
$\cl B(\ell^2(Y,H))$ with $M_n(\cl B(\ell^2(X,H))$, we have that $T = T_k$.
By Proposition \ref{pr:kandTkpositivetogether}, $k$ is positive definite.
Let $\psi : Y\times Y \to CB(A,\cl B(H))$ be given by $\psi(i,x,j,y) = \nph(x,y)$ for all $i,j\in [n]$ and all $x,y\in X$.
We claim that $\psi$ is of positive type.
Indeed, let $m\in \bb{N}$, $x_1,\dots,x_m\in X$,
and $(a_{i,p,j,q})_{i,p,j,q}\in M_{nm}(A)^+$. Set $C_{p,q} = (a_{i,p,j,q})_{i,j}\in M_n(A)$.
By Lemma \ref{r_pt} (i),
$$(\psi(i,x_p,j,x_q)(a_{i,p,j,q}))_{i,p,j,q} = (\nph^{(n)}(x_p,x_q)(C_{p,q}))\in M_{mn}(\cl B(H))^+.$$
By Lemma \ref{r_pt} (ii), $\psi\cdot k$ is positive definite.
In addition, $(T_{\nph\cdot k_{i,j}})_{i,j} = T_{\psi\cdot k}$, and another application of
Proposition \ref{pr:kandTkpositivetogether} shows that $(T_{\nph\cdot k_{i,j}})_{i,j}\in M_n(\cl B(H))^+$.
Thus, $S_{\nph}$ is completely positive on $\cl S_2(X\times X,A)$.

Set $C = \sup_{x\in X} \|\nph(x,x)\|_{\cb}$.
Let $E\subseteq X$ be a finite set, say $|E| = m$. We view $\ell^2(E)\otimes H$ as a subspace of $\ell^2(X)\otimes H$,
and $M_m\otimes A \equiv \cl S_2(E\times E,A) \subseteq \cl S_2(X\times X,A)$.
The map $S_{\nph}$ leaves $\cl S_2(E\times E,A)$ invariant and,
by the previous paragraph, its restriction $S_{\nph,E}$ to the latter space is completely positive.
Let $(e_i)_{i \in \mathcal{I}}$ be an approximate unit in $A$.
It is easy to check that the family $(\tilde{e}_i)_{i \in \mathcal{I}}$,
where $\tilde{e}_i$ is the `diagonal function' constantly equal to $e_i$,
is an approximate unit in the $C^*$-algebra  $\cl S_2(E\times E,A)$.
Since each of the maps $\nph(x,x)$ is completely positive, and for a completely positive map on a C*-algebra its completely bounded norm can be expressed as the limit of the norms of the images of an approximate unit (as follows from Stinespring theorem for non-unital completely positive maps, see for example Appendix A of \cite{NeshveyevStormer}), we have
\[
    \|S_{\nph,E}\|_{\cb} = \lim_{i \in \mathcal{I}} \|S_{\nph,E}(\tilde{e}_i)\| = \lim_{i \in \mathcal{I}} \max_{x\in E}\|\nph(x,x)(e_i)\| = \max_{x\in E}   \|\nph(x,x)\|_{\cb} \leq C.
\]
Since the spaces $(\cl S_2(E\times E,A))_{E\subseteq X}$ form an upwards directed net dense in $\cl S_2(X\times X,A)$,
where the indexing set
of all finite subsets of $X$ is given the inclusion order, we conclude that
$S_{\nph}$ is completely bounded on $\cl S_2(X\times X,A)$ and $\|S_{\nph}\|_{\cb}\leq C$.
Since the map $\nph(x,x)$ is a compression of  $S_{\nph}$,
we have that
$\|\nph(x,x)\|_{\cb}\leq \|S_{\nph}\|_{\cb}$, $x\in X$, and hence
$\|S_{\nph}\|_{\cb} = C$.
On the other hand, the density of $\cl S_2(X\times X,A)$ in $\cl K\otimes A$ implies that
the matricial cones of $\cl K\otimes A$ are closures of the corresponding cones of $\cl S_2(X\times X,A)$.
It follows that $S_{\nph} : \cl K\otimes A \to \cl K\otimes A$ is completely positive.

(ii)$\Rightarrow$(iii) follows the steps of the proof of \cite[Theorem 2.6]{mtt}, using the Stinespring theorem
instead of the Haagerup-Paulsen-Wittstock theorem. We leave the detailed verification to the interested reader.

(iii)$\Rightarrow$(i)
Let $x_1,\dots,x_m\in X$ and $a_{p,q}\in A$ be such that $C = (a_{p,q})_{p,q}\in M_{m}(A)^+$.
Letting
$V = \oplus_{p=1}^m V(x_p)$,
we have that
\begin{eqnarray*}
(\nph(x_p,x_q)(a_{p,q}))_{p,q}
& = &
((V(x_p)^*(\rho(a_{p,q})) V(x_q))_{p,q}\\
& = & V^*\rho^{(m)}(C)V \in M_{m}(\cl B(H))^+.
\end{eqnarray*}

The last statement follows
from the representation  \eqref{Phiform}, taking into account that in the case $S_{\nph}$ is completely positive we may choose $V = W$.
\end{proof}

\medskip

We fix a discrete group $G$, a Hilbert space $H$, a non-degenerate $C^*$-algebra $A\subseteq \cl B(H)$,
and a homomorphism
$\alpha : G\to {\rm Aut}(A)$; we thus have that $(A,G,\alpha)$ is a $C^*$-dynamical system.
We let $\ell^1(G,A)$ be the $*$-algebra of all summable functions $f : G\to A$.
We write $\cl H = \ell^2(G)\otimes H$ and identify $\cl H$ with $\ell^2(G,H)$.
We denote by $\lambda : G\to \cl B(\cl H)$, $t\to \lambda_t$, the unitary representation of $G$
given by $\lambda_t \xi(s) = \xi(t^{-1}s)$, $s\in G$, $\xi\in \cl H$,
and write $\lambda_t^0$ for the corresponding left regular unitary acting on $\ell^2(G)$.
Thus, $\lambda_t = \lambda_t^0 \otimes I$.
We also let $\pi : A\to \cl B(\cl H)$ be the $*$-representation given by
$\pi(a)\xi(s) = \alpha_{s^{-1}}(a)(\xi(s))$, $s\in G$, $a\in A$.
We note the covariance relation
\begin{equation}\label{eq_cov}
    \pi(\alpha_t(a)) = \lambda_t \pi(a) \lambda_t^*, \quad a\in A, t\in G.
\end{equation}
The pair $(\pi,\lambda)$ gives rise to a $*$-representation
$\tilde{\pi} : \ell^1(G,A) \to \cl B(\cl H)$ such that
\begin{equation}\label{eq_repel}
    \tilde{\pi}(f) = \sum_{s\in G} \pi(f(s))\lambda_s, \quad f\in \ell^1(G,A).
\end{equation}
(Note that the series on the right hand side of (\ref{eq_repel}) converges in norm for every $f\in \ell^1(G,A)$.)
The reduced crossed product $A\rtimes_{\alpha,r} G$ is defined by letting
\[
    A\rtimes_{\alpha,r} G = \overline{\tilde{\pi}(\ell^1(G,A))},
\]
where the closure is taken in the operator norm of $\cl B(\cl H)$.
Note that, after identifying $A$ with $\pi(A)$, we may consider $A$ as a $C^*$-subalgebra of $A\rtimes_{\alpha,r} G$.

Identifying $\cl H$ with $\oplus_{s\in G} H$, we associate to every operator $x\in \cl B(\cl H)$
a corresponding matrix $(x_{s,t})_{s,t \in G}$, where $x_{s,t}\in \cl B(H)$ (we use the standard identification: for
$\xi, \eta \in H$ and $s, t \in G$ we have
$\ip{ x(\delta_t \otimes \xi) }{ \delta_s \otimes \eta} = \ip{ x_{s,t} \xi}{ \eta}$). Note that if
$x\in A\rtimes_{\alpha,r} G$ then the diagonal of its matrix coincides with $(\alpha_{t^{-1}}(a_x))_{t\in G}$ for some $a_x\in A$.
This in particular implies that the transformation $\cl E$ that maps $x$ to $a_x$ is a conditional expectation from
$A\rtimes_{\alpha,r} G$ onto $A$.
We also consider the maps $\cl E_t : A\rtimes_{\alpha,r} G\to A$ given by $\cl E_t(x) = \cl E(x\lambda_t^*)$; note that
$\cl E_e = \cl E$.
Thus, to every operator $x\in A\rtimes_{\alpha,r} G$, one can associate the family $(a_t)_{t\in G}\subseteq A$, where
$a_t = \cl E_t(x)$; we write $x\sim \sum_{t\in G} \pii{a_t} \lambda_t$, although the series is formal and no convergence is assumed.
We note that
\begin{equation}\label{eq_el}
    \alpha_{t}(\cl E(x))=\cl E(\lambda_t x\lambda_t^*), \quad t\in G, x\in A\rtimes_{\alpha,r} G.
\end{equation}
The latter equality is straightforward in the case $x = \sum_{s\in F} \pii{a_s} \lambda_s$ for a finite subset $F\subseteq G$,
and follows by continuity for a general $x\in A\rtimes_{\alpha,r} G$.
Similarly, we can check that for any $x \in A\rtimes_{\alpha,r} G$ and $ p, q \in G$, we have $x_{p,q} \in A$ and
\begin{equation}\label{translationformula}
    x_{p,q} = \alpha_{p^{-1}} (x_{e, qp^{-1}}).
\end{equation}
Finally note that, as is well-known,
the construction of the reduced crossed product does not depend (up to a $*$-isomorphism) on the choice of the initial
faithful embedding $A \subseteq \cl B(H)$.

If $F : G\to CB(A)$ is a bounded map and $f\in \ell^1(G,A)$, let $F\cdot f\in  \ell^1(G,A)$ be the function given by
$$(F\cdot f)(t) = F(t)(f(t)), \ \ \ t\in G.$$
Recall \cite{mtt} that $F$ is called a Herz-Schur $(A,G,\alpha)$-multiplier
(or simply a Herz-Schur multiplier if the dynamical system is clear from the context)
if the map $S_F$, given by
$$S_F(\tilde{\pi}(f)) = \tilde{\pi}(F\cdot f), \ \ \ f\in \ell^1(G,A),$$
is completely bounded.
If $F$ is a Herz-Schur multiplier, then the map $S_F$ has a (unique) extension to a
completely bounded map on $A\rtimes_{\alpha,r} G$, which will be denoted in the same way.

For a bounded function $F : G\to CB(A)$, let
$\mathcal{N}(F): G \times G \to CB(A)$ be the function given by
\[ \mathcal{N}(F)(s,t) (a) = \alpha_{s^{-1}}(F(st^{-1})(\alpha_s(a))), \;\;\;\; s,t \in G, a \in A.\]
It was shown in \cite{mtt} that $\cl N$ is an isometric injection
from the algebra of all Herz-Schur $(A,G,\alpha)$-multipliers into
the algebra of all Schur $A$-multipliers.
We note that in \cite{mtt} a different (but similar) convention  was used for defining $\cl N(F)$;
however we found that for the purposes of positivity the definition given above is more natural.

\begin{definition}\label{de:positiveHSmult}
Let $(A,G,\alpha)$ be a $C^*$-dynamical system.
A Herz-Schur $(A,G,\alpha)$-multiplier $F$ will be called
\emph{completely positive}
if the map $S_F : A\rtimes_{\alpha,r} G \to A\rtimes_{\alpha,r} G$ is completely positive.
\end{definition}

\begin{theorem}\label{co:equivalencesofpositiveHSmultiplier}
Let $F : G\to CB(A)$ is a bounded function.
The following are equivalent:
\begin{enumerate}[(i)]
    \item $F$ is a completely positive Herz-Schur $(A,G,\alpha)$-multiplier;
    \item the function $\mathcal{N}(F)$ is of positive type;
    \item $\mathcal{N}(F)$ is a Schur $A$-multiplier and $S_{\mathcal{N}(F)}$ is completely positive.
\end{enumerate}
Moreover, if (i) holds then $\|S_F\|_{\cb} = \|S_{\mathcal{N}(F)}\|_{\cb} = \|F(e)\|_{\cb}$.
\end{theorem}
\begin{proof}
(i)$\Rightarrow$(ii)
Let $n\in \bb{N}$, $a_i\in A$ and $s_i\in G$, $i = 1,\dots,n$.
Using (\ref{eq_cov}), we obtain
\[
    (\pii{\alpha_{s_i}(a_i^*a_j)} \lambda_{s_i s_j^{-1}})_{i,j=1}^n
= (\lambda_{s_i}\pii{a_i^*a_j}\lambda_{s_j}^*)_{i,j=1}^n\in M_n(A\rtimes_{\alpha,r} G)^+ ,
\]
and hence
\[ \begin{split}
    X & \stackrel{\text{def}}{=} (\lambda_{s_i}\pii{\alpha_{s_i^{-1}}(F(s_i s_j^{-1})(\alpha_{s_i}(a_i^*a_j)))}\lambda_{s_j^{-1}})_{i,j=1}^n \\
    &= (\pii{F(s_i s_j^{-1})(\alpha_{s_i}(a_i^*a_j))}\lambda_{s_i s_j^{-1}})_{i,j=1}^n \in M_n(A\rtimes_{\alpha,r} G)^+.
\end{split} \]
Letting $\Lambda \in M_n(A\rtimes_{\alpha,r} G)$ be the diagonal matrix with diagonal entries
(in this order) $\lambda_{s_1^{-1}},\dots,\lambda_{s_n^{-1}}$,
we have that
\[
    (\pii{\alpha_{s_i^{-1}}(F(s_i s_j^{-1})(\alpha_{s_i}(a_i^*a_j)))})_{i,j=1}^n
= \Lambda X \Lambda^* \in M_n(A\rtimes_{\alpha,r} G)^+.
\]
Since every positive matrix in $M_n(A)$ is the sum of matrices of the form $(a_i^*a_j)_{i,j}$,
we have that $\cl N(F)$ is of positive type.

(ii)$\Rightarrow$(iii) follows from Theorem \ref{re:Sphicpintermsofphi}.

(iii)$\Rightarrow$(i)
It suffices to show that $S_{\cl N(F)}(\pii{a}\lambda_s)=\pii{F(s)(a)}\lambda_s$ for all $a\in A$ and $s\in G$.
Let $a\in A$ and $s\in G$.
Writing $(b_{p,q})_{p,q\in G}$ for the $A$-valued matrix of $\pii{a}\lambda_s$,
we have that 
\[
b_{p,q} =  \begin{cases}
                     \alpha_{p^{-1}}(a),  & \text{if } pq^{-1} = s; \\
                    0       & \text{otherwise.}
                \end{cases}
\]
Thus, $S_{\cl N(F)}(\pii{a}\lambda_s) = (c_{p,q})_{p,q\in G}$, where
\[
    c_{p,q} =  \begin{cases}
                     \cl N(F)(p,q)(\alpha_{p^{-1}}(a)),  & \text{if } pq^{-1} = s; \\
                    0       & \text{otherwise;}
                \end{cases}
\]
that is, $S_{\cl N(F)}(\pii{a}\lambda_s) =  \pii{F(s)(a)}\lambda_s$.

Finally, the equalities involving the norms follow from \cite[Theorem 3.8]{mtt} and Theorem~\ref{re:Sphicpintermsofphi}.
\end{proof}

\begin{corollary}\label{c_lircp}
Let $\nph : G\times G\to CB(A)$ be a Schur $A$-multiplier. The following are equivalent:
\begin{enumerate}[i.]
    \item $S_{\nph}$ is completely positive and leaves $A\rtimes_{\alpha,r} G$ invariant;
    \item $\nph = \cl N(F)$ for some completely positive Herz-Schur $(A,G, \alpha)$-multiplier  $F : G\to CB(A)$.
\end{enumerate}
\end{corollary}
\begin{proof}
(i)$\Rightarrow$(ii)
The fact that $S_{\nph}$ leaves $A\rtimes_{\alpha,r} G$ invariant shows that
if $s\in G$ then
for every $a\in A$ there exists $b\in A$ such that
\[
    \nph(p,q)(\alpha_{p^{-1}}(a)) = \alpha_{p^{-1}}(b), \quad\mbox{ whenever } pq^{-1} = s.
\]
Thus, the function $(p,q)\to \alpha_p(\nph(p,q)(\alpha_{p^{-1}}(a)))$ depends only on $pq^{-1}$.
Set $F(s)(a) = \alpha_p(\nph(p,q)(\alpha_{p^{-1}}(a)))$, where $p,q\in G$ are such that $pq^{-1} = s$.
Then $F : G\to CB(A)$ is well-defined and $\cl N(F) = \nph$.
Since $S_{\nph}$ is completely positive, $F$ is a completely positive Herz-Schur $(A,G, \alpha)$-multiplier.

(ii)$\Rightarrow$(i) follows from Theorem \ref{co:equivalencesofpositiveHSmultiplier}.
\end{proof}

\begin{remarks} 
{\rm
{\bf (i)} It follows from Theorem \ref{re:Sphicpintermsofphi} that if
$F : G\to CB(A)$ is a bounded function such that
the function $\cl N(F)$ is of positive type then
$F$ is automatically a Herz-Schur $(A,G,\alpha)$-multiplier.

\smallskip

\noindent {\bf (ii)} Suppose that $F : G\to CB(A)$ is a Herz-Schur $(A,G,\alpha)$-multiplier. The positivity of the map $S_F : A\rtimes_{\alpha,r} G \to A\rtimes_{\alpha,r} G$
does not imply its complete positivity. This becomes evident if one considers the case where $G = \{e\}$ and
$A$ is any $C^*$-algebra that admits positive maps that are not completely positive ({\it e.g.}\ $A = M_n$ for $n\geq 2$).}
\end{remarks}

\medskip

We next compare the notion of completely positive Herz--Schur $(A,G,\alpha)$-multipliers to other similar notions
that can be found in the literature.
Let $(A,G,\alpha)$ be a $C^*$-dynamical system.
\begin{itemize}
    \item Let us call a function $T : G \times A \to A$, which is linear in the second variable, \emph{positive definite in the sense of B\'edos-Conti}, or \emph{BC positive definite}, if for any $n \in \NN$, any $s_1 , \ldots , s_n \in G$, and any $a_1 , \ldots , a_n \in A$, the matrix
        \[
            \Big( \alpha_{s_i} \big( T(s_i\inv s_j , \alpha_{s_i\inv}(a_i^* a_j)) \big) \Big)_{i,j=1}^n
        \]
is a positive element of $M_n(A)$. This definition was given by B\'edos and Conti~\cite[Definition 4.7]{bc2} in the more general case of a twisted $C^*$-dynamical system; here we consider only the trivial twist and have simplified the definition accordingly.
    \item Let us call a function $h : G \to \mathcal{Z}(A)$ \emph{positive definite in the sense of Dong-Ruan}, or \emph{DR positive definite}, if, for any $n \in \NN$ and any $s_1 , \ldots , s_n \in G$, the matrix
            \[
                \Big( \alpha_{s_j} \big( h (s_i\inv s_j) \big) \Big)_{i,j=1}^n
            \]
        is a positive element of $M_n(A)$. This definition was given by Dong and Ruan~\cite[p.\ 436]{dong_ruan}; only centre-valued functions are considered because this is a necessary condition for such a map to be a `multiplier' of the reduced crossed product in the sense of \cite{dong_ruan}.
    \item Let us call a function $\varphi : G \to A$ \emph{$\alpha$-positive definite}, if for any $n \in \NN$ and any $s_1 , \ldots , s_n \in G$ the matrix
            \[
                \Big( \alpha_{s_i} \big( \varphi (s_i\inv s_j) \big) \Big)_{i,j=1}^n
            \]
        is a positive element of $M_n(A)$. This definition was given in \cite[D\'efinition 2.1]{Claire}, and used in \cite[p. 3]{bc2}.
\end{itemize}

\noindent
We comment on how the above notions compare to Definition~\ref{de:positiveHSmult}; we will not consider DR positive definiteness since it is similar to $\alpha$-positive definiteness.

One can easily show that a function $T : G \times A \to A$, which is linear in the second variable,
is a completely bounded multiplier of $(A,G,\alpha)$ in the sense of B\'edos--Conti~\cite{bc2} if and only if the function
\[
    F_T : G \to CB(A) ;\ F_T(t)(a) \stackrel{\text{def}}{=} T(t,a) , \quad t \in G,\ a \in A ,
\]
is a Herz--Schur $(A,G,\alpha)$-multiplier.
Let $n \in \NN$, $s_1 , \ldots , s_n \in G$, and $a_1 , \ldots , a_n \in A$;
then
\[ \begin{split}
    \Big( \alpha_{s_i\inv} \big( T(s_i s_j^{-1} , \alpha_{s_i}(a_i^* a_j) ) \big) \Big)_{i,j=1}^n
    &= \Big( \alpha_{s_i\inv} \big( F_T(s_i s_j\inv) \big( \alpha_{s_i}(a_i^* a_j) \big) \big) \Big)_{i,j=1}^n \\
    &= \big( \mathcal{N}(F_T)(s_i , s_j) (a_i^* a_j) \big)_{i,j=1}^n.
\end{split} \]
This implies that $T$ is BC positive definite if and only if $F_T$ is a Herz--Schur $(A,G,\alpha)$-multiplier of positive type, since any positive matrix in $M_n(A)$ is a sum of matrices of the form $(a_i^* a_j)_{i,j=1}^n$.

Now suppose that $\varphi : G \to \cl Z(A)$ is $\alpha$-positive definite. Let $F_\varphi : G \to CB (A)$ be given by
\[
    F_\varphi (t)(a) = \varphi(t) a , \quad t \in G,\ a \in A .
\]
Let $n \in \NN$, $s_1 , \ldots , s_n \in G$, and $(a_{i,j})$ a positive matrix in $M_n(A)$; then
\[ \begin{split}
    \big( \mathcal{N}(F_\varphi)(s_i,s_j)(a_{i,j}) \big)_{i,j=1}^n &= \Big( \alpha_{s_i\inv}\big( F_\varphi(s_i s_j\inv) \big( \alpha_{s_i}(a_{i,j}) \big) \big) \Big)_{i,j=1}^n \\
        &= \Big( \alpha_{s_i} \big( \nph(s_i s_j\inv) \big) a_{i,j} \Big)_{i,j=1}^n ,
\end{split} \]
which is positive as the Schur product of  a positive matrix in $M_n(\cl Z(A))$ and a positive matrix in $M_n(A)$ (as follows from an elementary calculation;
note that the argument would not be valid for general $A$-valued maps).
Conversely, if $\mathcal{N}(F_\varphi)$ is a Herz-Schur $(A,G,\alpha)$-multiplier of positive type then,
since the matrix in $M_n(A)$ with all entries equal to $\unA$ is positive, we have for any $n \in \NN$ and any $s_1 , \ldots , s_n \in G$,
\[
    0 \leq \Big( \mathcal{N}(F_\varphi)(s_i,s_j)(1_A) \Big)_{i,j=1}^n = \Big( \alpha_{s_i^{-1}} \big( \varphi (s_i s_j\inv) \big) \Big)_{i,j=1}^n ,
\]
which shows that $\varphi$ is an $\alpha$-positive definite function.

\bigskip

We finish this section by exhibiting a class of examples of positive Herz-Schur $(A,G,\alpha)$-multipliers.

\begin{proposition}\label{hsh}
Let $F$ be a finite subset of $G$ and $\Phi\in CB(A)$.
The map $h_F : G\to \cl B(A)$, given by
\[
    h_F(s)(a) = \sum_{p\in F\cap sF}\alpha_p\circ\Phi\circ\alpha_p^{-1}(a), \quad a\in A,
\]
is a Herz-Schur $(A,G,\alpha)$-multiplier with $\|S_{h_F}\|_{\rm cb}\leq |F| \|\Phi\|_{\rm cb}$.
Moreover, $h_F$ is a completely positive Herz-Schur $(A,G,\alpha)$-multiplier if the map $\Phi$ is completely positive.
\end{proposition}
\begin{proof}
We have
\[ \begin{split}
    \mathcal N(h_F)(s,t)(a) &= \alpha_{s^{-1}}(h_F(st^{-1})(\alpha_s(a))) =\sum_{p\in F\cap st^{-1}F}\alpha_{s^{-1}p}(\Phi(\alpha_{p^{-1}s}(a))) \\
    &= \sum_{p\in (s^{-1}F)\cap (t^{-1}F)}\alpha_p(\Phi(\alpha_{p}^{-1}(a))) \\
    &= \sum_{p\in G}\chi_{s^{-1}F}(p)\alpha_p(\Phi(\alpha_{p}^{-1}(a)))\chi_{t^{-1}F}(p)
\end{split} \]
(here we denote by $\chi_{E}$ the characteristic function of a set $E$).
As $\alpha_p\circ\Phi\circ\alpha_p^{-1}$ is a completely bounded map on $A$ for each $p\in G$, there exist a Hilbert space $H_p$, $*$-representation $\pi_p:A\to \cl B(H_p)$ and bounded operators $V_p,W_p\in B(H,H_p)$ such that
$$(\alpha_p\circ\Phi\circ\alpha_p^{-1})(a)=V_p^*\pi_p(a)W_p, \ \ \ a\in A,$$
and $\|\alpha_p\circ\Phi\circ\alpha_p^{-1}\|_{\rm cb}=\|\Phi\|_{\rm cb}=\|V_p\|^2=\|W_p\|^2$.

Let $\rho=\oplus_{p\in G}\pi_p:A\to\cl B(\oplus_{p\in G} H_p)$,
and $V(s), W(s) : H\to\oplus_{p\in G} H_p$ be the column operators given by
$V(s) = (\chi_{s^{-1}F}(p)V_p)_{p\in G}$ and $W(s) = (\chi_{s^{-1}F}(p)W_p)_{p\in G}$.
The latter are bounded operators with norms
\[
    \|V(s)\| = \left\|\sum_{p\in s^{-1}F}V_p^*V_p\right\|^{1/2}\leq |F|^{1/2} \|\Phi\|_{\rm cb}^{1/2}
\]
and
\[
    \|W(t)\| = \left\|\sum_{p\in t^{-1}F}W_p^*W_p\right\|^{1/2}\leq |F|^{1/2} \|\Phi\|_{\rm cb}^{1/2}.
\]
Moreover,
\[
    \mathcal N(h_F)(s,t)(a)=V(s)^* \rho(a) W(t), \ \ \ s, t\in G.
\]
Hence, by Theorem \ref{th_mtt}, $\mathcal N(h_F)$ is a Schur $A$-multiplier with $\|\mathcal N(h_F)\|_{\rm m}\leq |F|\|\Phi\|_{\rm cb}$.
By \cite[Theorem  3.8]{mtt}, $h_F$ is a Herz-Schur $(A,G,\alpha)$-multiplier.
If $\Phi$ is completely positive  then we can choose $V_p=W_p$  and hence $V(s)=W(s)$ for
every $s\in G$. In this case, by Theorem~\ref{re:Sphicpintermsofphi} and Theorem~\ref{co:equivalencesofpositiveHSmultiplier}, $h_F$ is a completely positive Herz-Schur $(A,G,\alpha)$-multiplier.
\end{proof}


\section{The Haagerup property}
\label{s_hp}

\noindent
In this section we let $A$ be a \emph{unital} $C^*$-algebra,
whose identity will be denoted by $\unA$, equipped with a faithful tracial state $\tau$.
We denote by $L^2(A,\tau)$ the completion of $A$ with respect to the
norm $\|a\|_{2,\tau} := (\tau(a^*a))^{1/2}$.
We say that a map $\Phi : A\to A$ is \emph{$L^2$-bounded} if there exists a constant $C > 0$ such that
$\|\Phi(a)\|_{2,\tau} \leq C\|a\|_{2,\tau}$ for every $a\in A$.
If this happens, there exists a (unique) bounded operator
$T_{\Phi} : L^2(A,\tau)\to L^2(A,\tau)$ with the property that $T_{\Phi}(a) = \Phi(a)$ whenever $a\in A$.
An $L^2$-bounded map $\Phi$ will be called \emph{$L^2$-compact} if $T_{\Phi}$ is compact.
An application of the Cauchy-Schwarz inequality shows that a unital completely positive map $\Phi$
satisfying the condition $\tau\circ\Phi\leq\tau$ is $L^2$-bounded and that $T_{\Phi}$ is a contraction.

The following definition is due to Dong~\cite{dong}.

\begin{definition}\label{d_hpcst}
The pair $(A,\tau)$, where $A$ is a unital $C^*$-algebra and $\tau$ is a faithful tracial state on $A$, is said
to possess the \emph{$C^*$-algebra Haagerup property}
if there exists a net $(\Phi_i)_{i\in I}$ of unital completely positive maps on $A$ such that
\begin{itemize}
  \item[(i)] $\tau \circ \Phi_i \leq \tau$ for all $i\in I$;
  \item[(ii)] $\Phi_i$ is $L^2$-compact for all $i\in I$;
  \item[(iii)] $\|\Phi_i(x) - x\|_{2,\tau}\to 0$, for all $x\in A$.
\end{itemize}
\end{definition}

Let $(A,G,\alpha)$ be a $C^*$-dynamical system, where the group $G$ is discrete,
$A$ has a faithful tracial state $\tau$, and $\alpha$ is $\tau$-preserving.
Following Dong \cite{dong},
we will consider the reduced crossed product $A \rtimes_{\alpha ,r} G$ endowed with the induced trace $\tau'$,
given by
\[
    \tau' (x) = \tau(\cl E(x)), \quad x\in A\rtimes_{\alpha,r} G.
\]

\begin{lemma}\label{l_autb}
Let $F : G\to CB(A)$ be a completely
positive Herz-Schur $(A,G,\alpha)$-multiplier such that $\tau\circ F(e)\leq \tau$ and $F(e)(\unA) = \unA$.
Then $F(t)$ is $L^2$-bounded and $T_{F(t)}$ is a contraction, for each $t\in G$.
\end{lemma}
\begin{proof}
Set $\nph = \cl N(F)$; thus, if $s,t\in G$ and $a\in A$ then
$\nph(s,t)(a) = \alpha_{s^{-1}}(F(st^{-1})(\alpha_s(a)))$.
Since $\nph$ is of positive type, it is hermitian.
Thus, if $s,t\in G$ and $b\in A$, we have
$\nph(t,s)(b^*) = \nph(s,t)(b)^*$, that is,
\[
    \alpha_{t^{-1}}(F(ts^{-1})(\alpha_t(b^*))) = \alpha_{s^{-1}}(F(st^{-1})(\alpha_s(b)))^*, \quad s,t\in G, b\in A.
\]
Letting $a = \alpha_t(b)$ and using the fact that $\alpha$ is a homomorphism, we
obtain
\[
    F(ts^{-1})(a^*) = \alpha_{ts^{-1}}(F(st^{-1})(\alpha_{st^{-1}}(a)))^*,
\]
and so
\[
    F(r)(a^*) = \alpha_{r}(F(r^{-1})(\alpha_{r^{-1}}(a)))^*, \quad r\in G, a\in A;
\]
the latter identity can also be rewritten as
\begin{equation}\label{eq_rst}
F(r)(a)^* = \alpha_r(F(r^{-1})(\alpha_{r^{-1}}(a)^*)), \quad r\in G, a\in A.
\end{equation}

Fix $a\in A$ and let $a_1 = \unA$ and $a_2 = a^*$.
In (\ref{eq_pt}),  set $m=2$, $n= 1$, $x_1 = e$, $x_2 = t$, $a^{p,q} = a_p a_q^*$, $p,q = 1,2$.
Since $\nph$ is of positive type, we have
\[
    \left(\begin{matrix}
        F(e)(\unA) & F(t\inv)(a) \\
        \alpha_{t^{-1}}(F(t)(\alpha_t(a^*))) & \alpha_{t^{-1}}(F(e)(\alpha_t(a^*a)))
    \end{matrix}\right) \in M_2(A)^+.
\]
Since $F(e)(\unA) = \unA$, and $\alpha_t$ is unital and completely positive,
\[
    \left(\begin{matrix}
        \unA & \alpha_t(F(t^{-1})(a)) \\
        F(t)(\alpha_t(a^*)) & F(e)(\alpha_t(a^*a))
    \end{matrix}\right)\in M_2(A)^+;
\]
writing $b = \alpha_t(a^*)$, we have
\[
    \left(\begin{matrix}
        \unA & \alpha_t(F(t^{-1})(\alpha_{t^{-1}}(b^*))) \\
        F(t)(b) & F(e)(bb^*)
    \end{matrix}\right) \in M_2(A)^+.
\]
In view of (\ref{eq_rst}),
\[
    \left(\begin{matrix}
        \unA & F(t)(b)^* \\
        F(t)(b) & F(e)(bb^*)
    \end{matrix}\right) \in M_2(A)^+
\]
and hence
\[
    F(t)(b) F(t)(b)^*\leq F(e)(bb^*) .
\]
It follows that
\[ \begin{split}
    \|F(t)(b)\|_{2,\tau}^2 &= \tau(F(t)(b)^* F(t)(b)) \leq \tau(F(e)(bb^*)) \\
    &\leq \tau(bb^*) = \tau(b^*b)= \|b\|_{2,\tau}^2,
\end{split} \]
for every $b\in A$;
thus, $F(t)$ is $L^2$-bounded and $T_{F(t)}$ is a contraction.
\end{proof}


In view of Lemma \ref{l_autb}, condition (i) in the following definition
implies that $F_i(s)$ is $L^2$-bounded, $s\in G$, $i\in I$, and hence it makes sense to
formulate conditions (ii) and (iii) in the subsequent definition.

\begin{definition}\label{de:Haageruppropertysystem}
Let $(A,G,\alpha)$ be a $C^*$-dynamical system,  where the group $G$ is discrete, $A$ has a faithful tracial state $\tau$,
and $\alpha$ is $\tau$-preserving. We say that $(A,G,\alpha, \tau)$ has the \emph{Haagerup property}
if there is a net $(F_i)_{i\in I}$ of completely positive Herz-Schur $(A,G,\alpha)$-multipliers such that
\begin{enumerate}[i.]
    \item $F_i(e)$ is unital and $\tau \circ F_i(e) \leq \tau$, $i\in I$;
    \item $F_i(t)$ is $L^2$-compact, $t \in G$;
    \item the function $s\to \|T_{F_i(s)}\|$ vanishes at infinity, $i\in I$;
    \item $\|F_i(t)(a) - a \|_{2,\tau} \to 0$ for all $t \in G$ and all $a \in A$.
\end{enumerate}
\end{definition}

We note that, in the case $A = \CC$,
Definition~\ref{de:Haageruppropertysystem} reduces to the definition of the Haagerup property for the group $G$.

Let $\Phi$ be a bounded linear map on $A\rtimes_{\alpha,r} G$,
and $h_\Phi : G\to \cl B(A)$ be the function given by
\begin{equation}\label{eq_hPhi}
    h_\Phi(s)(a) = \cl E(\Phi(\pii{a} \lambda_s)\lambda_s^*).
\end{equation}
The next proposition can be found in \cite{bc2}; for convenience of the reader we include a proof.

\begin{proposition}\label{p_cphs}
If $\Phi$ is a completely positive map on $A\rtimes_{\alpha,r} G$
then $h_{\Phi}$ is a completely positive Herz-Schur $(A,G,\alpha)$-multiplier.
\end{proposition}
\begin{proof}
By Stinespring's Theorem,
there exist a Hilbert space $K$, an operator $V : \cl H \to K$ and a $*$-representation
$\rho : A\rtimes_{\alpha,r} G \to \cl B(K)$ such that
\[
    \Phi(x)=V^*\rho(x)V, \quad x\in A\rtimes_{\alpha,r} G.
\]
Using (\ref{eq_el}), we have
\[ \begin{split}
\cl N(h_{\Phi})(s,t)
&=
    \alpha_{s^{-1}}(h_{\Phi}(st^{-1})(\alpha_s(a))) = \alpha_{s^{-1}}(\cl E(\Phi(\pii{\alpha_s(a)}\lambda_{st^{-1}})\lambda_{st^{-1}}^*)) \\
        &= \alpha_{s^{-1}}(\cl E(\Phi(\lambda_s \pii{a}\lambda_{s^{-1}}\lambda_{st^{-1}})\lambda_{st^{-1}}^*)) \\
        &= \cl E(\lambda_{s^{-1}} \Phi(\lambda_s \pii{a}\lambda_{t^{-1}})\lambda_t\lambda_{s^{-1}})\lambda_s)) \\
        &= \cl E(\lambda_{s^{-1}} \Phi(\lambda_s \pii{a} \lambda_{t^{-1}})\lambda_t)) \\
        &= \cl E(\lambda_s^* V^* \rho(\lambda_s)\rho(\pii{a})\rho(\lambda_t)^* V \lambda_t)) .
\end{split} \]
Letting $V(s)=\rho(\lambda_s)^*V\lambda_s$, $s\in G$, we obtain
\[
    \cl N(h_{\Phi})(s,t) = \cl E(V(s)^*\rho(\pii{a})V(t))), \quad a\in A, s,t\in G.
\]
It follows that $h_{\Phi}$ is a function of positive type.
\end{proof}

We will denote by $L^2(\tau)$ (resp. $L^2(\tau')$) the Hilbert space $L^2(A,\tau)$
(resp. $L^2(A\rtimes_{\alpha,r} G,\tau')$).
For $t\in G$, let $L^2_t(\tau')$ be the closure, in the norm $\|\cdot\|_{2,\tau'}$,
of the subspace
\begin{equation}\label{eq_lt}
    \cl L_t := \{ \pii{a}\lambda_t : a\in A \}
\end{equation}
of $L^2(\tau').$

\begin{lemma}\label{l_ds}
The following statements hold true.
\begin{enumerate}[(i)]
    \item We have an orthogonal decomposition
        \begin{equation}\label{eq_dec}
            L^2(\tau') = \oplus_{t\in G} L^2_t(\tau').
        \end{equation}
    \item For each $t \in G$, the map $a\to \pii{a}\lambda_t$ extends to an isometry $V_t$ from $L^2(\tau)$ to $L^2(\tau')$,
such that $V_t(L^2(\tau)) = L^2_t(\tau')$.
    \item Let  $P_t$ the orthogonal projection of $L^2(\tau')$ onto $L^2_t(\tau')$.
Then the map $V_t^* P_t : L^2(\tau') \to L^2(\tau)$
satisfies
\[
    V_t^* P_t (z) = \cl E (z \lambda_t^*), \quad  z \in A\rtimes_{\alpha,r} G.
\]
\end{enumerate}
\end{lemma}
\begin{proof}
(i) and (ii).
If $a, b \in A$ and $s, t \in G$ then
\[
    \langle \pii{a}\lambda_s, \pii{b}\lambda_t\rangle = \tau'(\lambda_t^*\pii{b^*a}\lambda_s) = \tau'(\pii{b^*a}\lambda_{st^{-1}})
= \tau(\cl E(\pii{b^*a}\lambda_{st^{-1}})).
\]
It follows that $\cl L_s\perp \cl L_t$ whenever $s\neq t$, and that
the map $a\to \pii{a}\lambda_t$ extends to a unitary from $L^2(\tau)$ onto $L^2_t(\tau')$.
The claim follows from the fact that $\cup_{t\in G} \cl L_t$ is total in $L^2(\tau')$.

(iii) follows after a simple check for $z = \pii{a} \lambda_s$, combined with the continuity of the involved maps.
\end{proof}

In the next lemma, we use Lemma \ref{l_ds} to identify $L^2(\tau')$ with $\oplus_{t\in G} L^2(\tau)$.

\begin{lemma}\label{l_dop}
Let $F : G\to CB(A)$ be a completely
positive Herz-Schur $(A,G,\alpha)$-multiplier such that $\tau\circ F(e)\leq \tau$ and $F(e)(\unA) = \unA$.
Then $S_F$ is $L^2$-bounded and $T_{S_F} = \oplus_{t\in G} T_{F(t)}$. In particular, $T_{S_F}$ is a contraction.
\end{lemma}
\begin{proof}
Clearly, $S_F$ leaves the space $\cl L_t$, defined in (\ref{eq_lt}), invariant.
After identifying $L_t^2(\tau')$ with $L^2(\tau)$ by Lemma~\ref{l_ds},
we have that the restriction of $S_F$ to $\cl L_t$ coincides with $F(t)$.
By Lemma~\ref{l_autb}, $\oplus_{t\in G} T_{F(t)}$ is a well-defined contraction.
Since $T_F$ coincides with the latter operator on a dense subspace, 
we conclude that $S_F$ is $L^2$-bounded and
$T_{S_F} = \oplus_{t\in G} T_{F(t)}$.
\end{proof}

\begin{lemma}\label{l_l2conv}
Let $(F_i)_{i\in I}$ be a net of completely
positive Herz-Schur $(A,G,\alpha)$ -multipliers such that
$F_i(e)$ is unital and $\tau\circ F_i\leq \tau$ for each $i$.
The following are equivalent:
\begin{enumerate}[(i)]
    \item $\|S_{F_i}(x) - x\|_{2,\tau'} \to 0$, $x\in A\rtimes_{\alpha,r} G$;
    \item $\|F_i(t)(a) - a \|_{2,\tau} \to 0$ for all $t \in G$ and all $a \in A$.
\end{enumerate}
\end{lemma}
\begin{proof}
For $s\in G$ and $a\in A$ we have, by Lemma \ref{l_ds} (ii),
\[
    \|\pii{F_i(s)(a)}\lambda_s - \pii{a}\lambda_s\|_{2,\tau'} = \|F_i(s)(a) - a\|_{2,\tau}, \quad i\in I.
\]
The equivalence now follows easily from Lemma~\ref{l_dop} and the fact that,
by Lemma~\ref{l_dop}, $\|T_{S_{F_i}}\|\leq 1$ for all $i\in I$.
\end{proof}

Now we can characterise the Haagerup property for $A \rtimes_{\alpha ,r} G$.

\begin{theorem}\label{th:Haagerupproperty}
Let $(A,G,\alpha)$ be a $C^*$-dynamical system,
where the group $G$ is discrete, $A$ has a faithful tracial state $\tau$ and $\alpha$ is $\tau$-preserving.
The following are equivalent:
\begin{enumerate}[(i)]
    \item $(A,G,\alpha, \tau)$ has the Haagerup property;
    \item $(A \rtimes_{\alpha ,r} G, \tau')$ has the $C^*$-algebra Haagerup property.
\end{enumerate}
\end{theorem}
\begin{proof}
(i)$\Rightarrow$(ii) Let $(F_i)_{i\in I}$ be a net of completely positive Herz-Schur
$(A,G,$ $\alpha)$-multipliers satisfying conditions (i)--(iv) of Definition~\ref{de:Haageruppropertysystem}.
Since $F_i(e)(\unA)$ $= \unA$, the map $S_{F_i}$ is unital.
By definition, $S_{F_i}$ is completely positive.
By Lemma~\ref{l_dop}, $T_{S_{F_i}}$ is a compact contraction for all $i\in I$.
By Lemma~\ref{l_l2conv}, $(T_{S_{F_i}})_{i\in I}$ converges to the identity operator in the strong operator topology.

(ii)$\Rightarrow$(i) Let $(\Phi_i)_{i\in I}$ be a net associated with
the $C^*$-algebra Haagerup property of $A \rtimes_{\alpha ,r} G$ as in Definition~\ref{d_hpcst}.
Write $F_i = h_{\Phi_i}$, $i\in I$. By Proposition~\ref{p_cphs}, $F_i$ is a completely positive Herz-Schur $(A,G,\alpha)$-multiplier.
Since $\Phi_i$ is unital, $F_i(e)(\unA) = \unA$.
Moreover, if $a\in A^+$ then
\[
    (\tau \circ F_i(e))(a) = \tau(\mathcal{E}(\Phi_i(\pii{a}))) = \tau' (\Phi_i(\pii{a})) \leq \tau'( \pii{a}) = \tau(a) ;
\]
that is, $\tau\circ F_i(e) \leq \tau$, $i\in I$.

Fix $t \in G$ and $a\in A$.
By the Cauchy-Schwarz inequality (see {\it e.g.} \cite[Proposition 3.3]{paulsen}),
\[ \begin{split}
    \| F_i(t)(a) \|_{2,\tau}^2
    &= \tau \Big( \mathcal{E} (\Phi_i(\pii{a}\lambda_t) \lambda_t^* )^* \big( \mathcal{E} (\Phi_i(\pii{a}\lambda_t) \lambda_t^* ) \big) \Big) \\
        &\leq \tau \Big( \mathcal{E} \big( (\Phi_i(\pii{a}\lambda_t) \lambda_t^* )^* (\Phi_i(\pii{a}\lambda_t) \lambda_t^* ) \Big) \\
        & = \tau' \Big( (\Phi_i(\pii{a}\lambda_t) \lambda_t^* )^* (\Phi_i(\pii{a}\lambda_t) \lambda_t^* ) \Big)
        = \| \Phi_i (\pii{a} \lambda_t) \|_{2,\tau'}^2 \\
        & = \| T_{\Phi_i} (\pii{a} \lambda_t) \|^2_{2,\tau'}.
\end{split} \]
Since $T_{\Phi_i}$ is compact,
\[
    \{ \Phi_i( \pii{a}\lambda_t ) : \|\pi(a)\|_{2,\tau'}\leq 1\}
\]
is a relatively compact set.
Since $\|\pii{a}\|_{2,\tau'} = \|a\|_{2,\tau}$, $a\in A$,
it follows that the subset $\{F_i(t)(a) : \|a\|_{2,\tau}\leq 1\}$ of $L^2(\tau')$
is relatively compact, and hence $F_i(t)$ is $L^2$-compact, $i\in I$.

Letting $P_t : L^2(\tau')\to L^2_t(\tau')$ be the orthogonal projection, we have that
\begin{equation}\label{eq_compre}
    T_{F_i(t)} = V_t^* P_t T_{\Phi_i} P_t V_t;
\end{equation}
indeed, for $a \in A$, by Lemma \ref{l_ds} (iii), we have
\[ \begin{split}
    T_{F_i(t)} (a) &= F_i(t)(a) = h_{\Phi_i}(t) (a)= \cl E (\Phi_i(\pii{a} \lambda_t) \lambda_t^*) \\
    &= V_t^* P_t \Phi_i(\pii{a} \lambda_t) = V_t^* P_t T_{\Phi_i}(\pii{a} \lambda_t) =
V_t^* P_t T_{\Phi_i} P_t V_t (a).
\end{split} \]
Since $V_t$ is an isometry, $T_{\Phi_i}$ is a compact operator, and $P_s\to_{s\to \infty} 0$ in the strong operator topology,
we have that $\|T_{F_i(t)}\|\to_{t\to \infty} 0$.
On the other hand, identity (\ref{eq_compre})
and the uniform boundedness of  the net $(T_{F_i(t)})_{i\in I}$
show that
$T_{F_i(t)}\to_{i} \id$ in the strong operator topology, and the proof is complete.
\end{proof}

The following result is well-known
(see for example \cite{Meng}); we next show how one can deduce it quickly from our Theorem~\ref{th:Haagerupproperty}.

\begin{corollary}\label{c_haageo}
Let $(A,G,\alpha)$ be a $C^*$-dynamical system,
where the group $G$ is discrete, $A$ has a faithful tracial state $\tau$ and $\alpha$ is $\tau$-preserving.
Assume that  $(A \rtimes_{\alpha ,r} G, \tau')$ has the $C^*$-algebra Haagerup property.
Then $G$ has the Haagerup property and $(A, \tau)$ has the $C^*$-algebra Haagerup property.
\end{corollary}
\begin{proof}
By Theorem \ref{th:Haagerupproperty},
there exists a net $(F_i)_{i\in I}$ of positive Herz-Schur $(A,G,\alpha)$-multipliers satisfying the conditions
of Definition~\ref{de:Haageruppropertysystem}.
Set
\begin{align*}
    T_i &= F_i(e), \quad i \in I, \\
    \varphi_i(s) &= \tau(F_i(s)(\unA)), \quad i \in I, s \in G.
\end{align*}
The fact that the net $(T_i)_{i\in I}$ yields the $C^*$-algebra Haagerup property for $A$
is a direct consequence of Definitions~\ref{d_hpcst} and \ref{de:Haageruppropertysystem}.
It thus suffices to check that  $(\varphi_i)_{i \in I}$ is a net of
normalised positive definite functions on $G$ vanishing at infinity and converging pointwise to $1$.
The fact that $\varphi_i$ is normalised is immediate from the unitality of $F_i(e)$, $i\in I$.
If $s \in G$ then
\[
    |\nph_i(s)| = |\ip{ F_i(s)(\unA)}{ \unA}_{2,\tau} | \leq \| \unA \|_{2,\tau} \| T_{F_i(s)}\|  \| \unA \|_{2,\tau} = \|T_{F_i(s)}\|,
\]
and property (iii) in Definition~\ref{de:Haageruppropertysystem} shows that $\varphi_i(s)\to 0$ as $s\to \infty$.
Further,
\[
    |\varphi_i(s) - 1| = |\ip{F_i(s)(\unA) - \unA}{\unA}_{2,\tau}| \leq  \| \unA \|_{2,\tau} \| F_i(s) (\unA) - \unA \|_{2, \tau}
\]
and so, by condition (iv) in Definition~\ref{de:Haageruppropertysystem},
$\varphi_i(s) \to 1$.
Finally, the positive definiteness of $\varphi_i$ can be checked by a standard matrix computation:
given $n \in \NN$, $s_1, \ldots, s_n \in G$ and $\mu_1, \ldots, \mu_n \in \mathbb{C}$, we have 
\[ \begin{split}
    \sum_{k,l=1}^n \overline{\mu_k} \varphi(s_k^{-1} s_l) \mu_l &= \tau' \left(\sum_{k,l=1}^n  \overline{\mu_k} S_{F_i} (\lambda_{s_k^{-1} s_l}) \mu_l \lambda_{s_l\inv s_k} \right) \\
    &=\tau' \left(\sum_{k,l=1}^n  \overline{\mu_k}  \lambda_{s_k}S_{F_i} (\lambda_{s_k^{-1} s_l}) \mu_l \lambda_{s_l^{-1}} \right),
\end{split} \]
where we used the trace property of $\tau'$.
Now note that
\[
    \sum_{k,l=1}^n  \overline{\mu_k}  \lambda_{s_k}S_{F_i} (\lambda_{s_k^{-1} s_l}) \mu_l \lambda_{s_l^{-1}}
\]
coincides with the operator 
\[
    \begin{bmatrix} \overline{\mu_1} \lambda_{s_1} & \cdots & \overline{\mu_n} \lambda_{s_n} \end{bmatrix}
S_{F_i}^{(n)}\left([\lambda_{s_k^{-1} s_l}]_{k,l=1}^n \right)
\begin{bmatrix} \overline{\mu_1} \lambda_{s_1} & \cdots & \overline{\mu_n} \lambda_{s_n} \end{bmatrix}^*,
\]
which is positive since $S_{F_i}^{(n)}$ is so.
\end{proof}

\begin{remarks}
{\rm {\bf (i)}
Note that the converse of Corollary \ref{c_haageo}, {\it i.e.} the statement that if $(A,\tau)$ has the $C^*$-algebra Haagerup property and $G$ has the Haagerup property then
$(A \rtimes_{\alpha ,r} G, \tau')$ has the $C^*$-algebra Haagerup property, which was claimed to hold in \cite{Meng}, is false: for example,
both ${\rm SL}(2, \Z)$ and $\Z^2$ have the Haagerup property but $\Z^2 \rtimes {\rm SL}(2,\Z)$ does not (see the references given in \cite[pg.\ 359, 374]{bo}).
The error originates in the earlier article \cite{You}: the maps $\Phi_n$ considered at the beginning of Section~2 of \cite{You} need not be completely positive.

\smallskip

\noindent {\bf (ii)} Dong and Ruan~\cite{dong_ruan} study a Hilbert module version of the Haagerup approximation property for crossed products and show that the relevant approximating maps --- which are required to be module maps --- can always be constructed
{\it via} multipliers associated to maps of the form $F(g)(a) = h(g)(a)$ ($g \in G, a \in A$), where $h:G \to \mathcal{Z}(A)$. It is clear that the
$A$-modularity condition on the maps $S_F$, together with the positivity requirement, force $F$ to be of the type described above, making the potential class of approximants much narrower compared to the one considered here.
On the other hand, as the approximation property of Dong and Ruan requires only the compactness of the maps $S_F$ with respect to $A$, and not as Hilbert space operators, the conditions on the multipliers they study do not involve any trace on $A$, which we need to impose in order to obtain the genuine $C^*$-algebra Haagerup property.}
\end{remarks}


\section{Nuclearity}
\label{s_n}

\noindent
In this section, we provide a characterisation of the nuclearity of $A\rtimes_{\alpha,r}G$ in terms of Herz-Schur $(A,G,\alpha)$-multipliers.
Recall that a $C^*$-algebra $B$  is called \emph{nuclear} if, for any other $C^*$-algebra $A$, the algebraic tensor product $A \odot B$ admits a unique $C^*$-norm.
It is well-known (and due to Choi, Effros, and Kirchberg, see Theorem~3.8.7 in \cite{bo}) that $B$ is nuclear if and only if there exist a net $(k_i)_{i\in I}$ of natural numbers
and completely positive contractions $\nph_i : B\to M_{k_i}(\bb{C})$
and $\psi_i : M_{k_i}(\bb{C}) \to B$ (which moreover can be assumed to be unital)
such that $\|(\psi_i\circ \nph_i)(x) - x\|\to 0$ for every $x\in B$.
Let $\cl F(B)$ denote the set of linear maps on $B$ of finite rank; note that \cite{ChoiEffros} shows that $B$ is nuclear if and only if there exists a net of completely positive contractions in $\cl F(B)$ approximating the identity map on $B$ pointwise in norm.

\begin{definition}\label{de:nuclearsystem}
A $C^*$-dynamical system $(A,G,\alpha)$, where $G$ is a discrete group and $A$ is a $C^*$-algebra,
will be called \emph{nuclear} if
there exists a net $(F_i)_{i\in I}$ of completely
positive, finitely supported, Herz-Schur $(A,G,\alpha)$-multipliers,
such that
\begin{enumerate}[i.]
    \item $\|F_i(e)\|_{\cb}\leq 1$ for all $i\in I$,
    \item $F_i(s)\in \cl F(A)$ for all $s\in G$ and all $i\in I$, and
    \item $\|F_i(s)(a)-a\|\to_{i\in I} 0$ for all $s\in G$ and all $a\in A$.
\end{enumerate}
\end{definition}

For the next lemma recall the definition of the maps $h_{\Phi}$ from (\ref{eq_hPhi}).

\begin{lemma}\label{le:equivalentconvergence}
(i) Let $(F_i)_{i\in I}$ be a net of completely positive Herz-Schur $(A,G,\alpha)$-multipliers
such that $\sup_i\|F_i(e)\| < \infty$.  Then $\|F_i(s)(a)-a\|\to_{i\in I} 0$ for all $s\in G$ and $a\in A$ if and only if
$\|S_{F_i}(x)-x\|\to_{i\in I} 0$  for all $x\in A\rtimes_{\alpha,r}G$.

(ii) If $(\Phi_i)_{i\in I}$ is a net of bounded linear maps on $A\rtimes_{\alpha,r} G$ such that
$\|\Phi_i(x)-x\| \to_{i\in I} 0$ for every $x\in A\rtimes_{\alpha,r} G$
then $\|h_{\Phi_i}(s)(a)-a\|\to_{i\in I} 0$ for every $a\in A$ and every $s\in G$.
\end{lemma}

\begin{proof}
(i) Assume that $\|F_i(s)(a)-a\|\to_{i\in I} 0$ for all $s\in G$ and $a\in A$.
Since
\[
    \|S_{F_i}\|_{\rm cb} = \|S_{F_i}(1)\| = \|F_i(e)(\unA)\|
\]
and $\sup_i\|F_i(e)\| < \infty$, it suffices to show that
$\|S_{F_i}(x) - x\| \to_{i\in I} 0$  whenever
$x = \sum_{s\in E} \pii{a_s} \lambda_s$ for a finite set $E\subseteq G$.
But, in this case,
\[
    \|S_{F_i}(x)-x\| = \left\| \sum_{s\in E} \pii{F_i(s)(a) - a} \lambda_s \right\| \leq \sum_{s\in E} \|F_i(s)(a) - a\| \to_{i\in I} 0.
\]

Conversely, suppose that $\|S_{F_i}(x)-x\| \to_{i\in I} 0$ for all $x\in A\rtimes_{\alpha,r} G$.
Then for fixed $s\in G$ and $a\in A$, we have
\[
    \|F_i(s)(a) - a\| = \|S_{F_i}(\pii{a} \lambda_s) - \pii{a} \lambda_s\|\to_{i\in I} 0.
\]

(ii) Assume that $\|\Phi_i(x)-x\| \to_i 0$ for all $x\in A\rtimes_{\alpha,r} G$. Then
\[ \begin{split}
    \|h_{\Phi_i}(s)(a) - a\| &= \|\cl E(\Phi_i(\pii{a} \lambda_s)\lambda_s^*) - \cl E((\pii{a}\lambda_s)\lambda_s^*)\| \\
    &\leq \|\Phi_i(\pii{a} \lambda_s) - \pii{a} \lambda_s\| \to_{i\in I} 0.
\end{split} \]
\end{proof}

We are now ready to state and prove the main result of this section.

\begin{theorem}\label{char_am}
Let $G$ be a discrete group and $(A,G,\alpha)$ be a $C^*$-dynamical system.
The following are equivalent:
\begin{enumerate}[(i)]
    \item $(A,G,\alpha)$ is nuclear;
    \item $A\rtimes_{\alpha,r} G$ is nuclear.
\end{enumerate}
\end{theorem}
\begin{proof}
(ii)$\Rightarrow$(i)
Assume that the $C^*$-algebra $B := A\rtimes_{\alpha, r}G$ is nuclear.
Then there exist a net $(k_i)_{i\in I}\subseteq \bb{N}$
and unital completely positive maps
$\varphi_i : B\to M_{k_i}(\mathbb C)$ and $\psi_i : M_{k_i}(\mathbb C)\to B$ such that $\psi_i\circ\varphi_i\to_{i\in I}\id_B$ in point-norm topology.
Let
\[
    B_0 = \left\{\sum_{s\in F} \pii{a_s} \lambda_s : F\subseteq G \text{ is finite}, a_s\in A, s\in F\right\}.
\]
We first show that there exists a net
$\tilde\psi_i : M_{k_i}(\mathbb C)\to B$ of completely positive contractions such that the range of
$\tilde \psi_i$ is in $B_0$ and
$\tilde \psi_i\circ\varphi_i\to\id_B$ in the point-norm topology.
Since $\psi_i$ is completely positive, and the matrix $(E_{p,q})_{p,q=1}^{k_i} \in M_{k_i} (M_{k_i}(\mathbb C))$ is positive, where $\{E_{p,q}:p,q=1,\ldots, k_i\}$ is the canonical matrix unit system of $M_{k_i}(\mathbb C)$, we have
$(\psi_i(E_{p,q}))_{p,q=1}^{k_i}\in M_{k_i}(B)^+$.
As $M_{k_i}(B_0)^+$ is dense in $M_{k_i}(B)^+$, given $\varepsilon_i>0$, $i\in I$, $\varepsilon_i\to_{i\in I} 0$,
there exists a matrix $(b_{p,q}^{i})_{p,q=1}^{k_i}\in M_{k_i}(B_0)^+$ such that
\[
    \|(\psi_i(E_{p,q}))_{p,q=1}^{k_i} - (b_{p,q}^{i})_{p,q=1}^{k_i}\| < \varepsilon_i/ k_i^3.
\]
Let $\psi_i' : M_{k_i}(\mathbb C)\to B$ be the map given by
\[
    \psi_i'((\lambda_{p,q})_{p,q=1}^{k_i}) = \sum_{p,q = 1}^{k_i} \lambda_{p,q} b_{p,q}^i.
\]
By Choi's Theorem~\cite[Theorem 3.14]{paulsen}, $\psi_i'$ is completely positive; moreover,
\[ \begin{split}
    |\|\psi_i'\|_{\cb} - 1| &= |\|\psi_i'(1)\| - \|\psi_i(1)\||  \leq \|\psi_i(1) - \psi_i'(1)\| \\
    &\leq  \sum_{p=1}^{k_i}\|\psi_i(E_{p,p}) - b^i_{p,p}\|\leq \varepsilon_i/k_i^2 .
\end{split} \]
Thus, $\|\psi_i'\|_{\cb}\to_{i\in I} 1$.

Let $\tilde\psi_i=\psi_i'/\|\psi_i'\|_{\cb}$.
Then $\tilde\psi_i$ is a completely  positive contraction. Moreover, for each $p,q=1,\ldots,k_i$, we have
\[ \begin{split}
    \|\psi_i(E_{p,q})-\tilde\psi_i(E_{p,q})\|
    &\quad \leq \frac{\|\psi_i(E_{p,q})-\psi_i'(E_{p,q})\|+|1-\|\psi_i'\|_{\cb}|\|\psi_i(E_{p,q})\|}{\|\psi_i'\|_{\cb}} \\
    &\quad \leq \left(\frac{\varepsilon_i}{k_i^3}+\frac{\varepsilon_i}{k_i^2}\right)/\|\psi_i'\|_{\cb} ,
\end{split} \]
and estimating very roughly we obtain for every matrix $x=(\lambda_{p,q})_{p,q=1}^{k_i} \in M_{k_i}(\mathbb C)$
\[ \begin{split}
    \|(\psi_i-\tilde\psi_i)((\lambda_{p,q})_{p,q})\| &=
    \left\|\sum_{p,q = 1}^{k_i}  \lambda_{p,q} (\psi_i-\tilde\psi_i)(E_{p,q})\right\|\\
& \leq \sum_{p,q = 1}^{k_i} |\lambda_{p,q}| \|(\psi_i-\tilde\psi_i)(E_{p,q})\|  \\
    &\leq \left(\frac{\varepsilon_i}{k_i^3} + \frac{\varepsilon_i}{k_i^2}\right)\sum_{p,q = 1}^{k_i} \frac{|\lambda_{p,q}|}{\|\psi_i'\|_{\cb}} \\
    &\leq \left(\frac{\varepsilon_i}{k_i} + \varepsilon_i\right) \frac{\|(\lambda_{p,q})_{p,q}\|}{\|\psi_i'\|_{\cb}} ,
\end{split} \]
giving
\[
    \|\psi_i - \tilde\psi_i\|\leq \frac{2\varepsilon_i}{\|\psi_i'\|_{\cb}}.
\]
Hence, for each $x\in B$,
\[ \begin{split}
    \|(\tilde\psi_i\circ\varphi_i) (x)-x\| &\leq \|(\tilde\psi_i-\psi_i)\circ\varphi_i(x)\|+\|\psi_i\circ\varphi_i(x)-x\| \\
    &\leq \|\tilde\psi_i-\psi_i\|\|x\|+\|(\psi_i\circ\varphi_i)(x)-x\|\to_{i\in I} 0 .
\end{split} \]

Let $\Phi_i = \tilde\psi_i\circ\varphi_i$ and $h_i = h_{\Phi_i}$, $i\in I$.
Note that, as the rank of $\Phi_i$ is finite, there exists a finite set $F_i \subseteq G$ such that,
for each $s\in G$ and $a \in A$, we have $\Phi_i(\pii{a}\lambda_s) = \sum_{t\in F_i} \pii{a_t} \lambda_t$.
By Proposition~\ref{p_cphs}, $h_i$ is a completely positive Herz-Schur $(A,G,\alpha)$-multiplier such that $\|h_i(e)\|_{\cb}\leq 1$.
As $h_i(s)(a) = \cl E(\Phi_i(\pii{a}\lambda_s)\lambda_s^*)$, and the range of $\Phi_i$ is finite dimensional, so is the range of $h_i(s)$ for each $s \in G, i \in I$.
Further, as for any $s, t \in G$ we have $\cl E(\pii{a_t} \lambda_t \lambda_s^*) = a_t$ for $s = t$ and zero otherwise, we obtain that $h_i$ is finitely supported (on the set $F_i$).
Moreover,
\[
    \|h_i(s)(a) - a\| = \|\cl E((\Phi_i(\pii{a} \lambda_s)\lambda_s^*)) - \cl E(\pii{a}\lambda_s\lambda_s^*)))\|\to_{i\in I} 0
\]
for all $a\in A$ and all $s\in G$.

(i)$\Rightarrow$(ii)
Let $(F_i)_{i\in I}$ be a net satisfying the conditions of Definition \ref{de:nuclearsystem}.
By Theorem~\ref{co:equivalencesofpositiveHSmultiplier},
the map $\Phi_i$ on $A\rtimes_{\alpha,r} G$, given by
\[
    \Phi_i\left(\sum_{s\in E} \pii{a_s} \lambda_s\right) = \sum_{s\in E} \pii{F_i(s)(a_s)}\lambda_s ,
\]
where $E \subseteq G$ is finite, is completely positive and $\|\Phi_i\|_{\cb} = \|F_i(e)\|\leq 1$.
As $F_i$ is finitely supported and $F_i(s) \in \cl F(A)$ for each $s\in G$,
the range of $\Phi_i$ is finite dimensional.
Since $\|\Phi_i(x) - x\|\to_{i\in I} 0$ for $x = \sum_{s\in E} \pii{a_s} \lambda_s$ with $E$ a finite set,
the uniform boundedness of the net $(\Phi_i)_{i\in I}$ shows that
$\|\Phi_i(x) - x\|\to_{i\in I} 0$ for every $x\in A\rtimes_{\alpha,r} G$.
Thus, $A\rtimes_{\alpha,r} G$ is nuclear.
\end{proof}

\begin{remark} \label{rem:nuclear}
{\rm
It follows from the above theorem that if $(A,G,\alpha)$ is a nuclear C*-dynamical system then $A$ is nuclear.
This can be seen directly: the net $(F_i(e))_{i \in I}$ is a sequence of completely positive, contractive, finite rank maps on $A$ converging pointwise to the identity.}
\end{remark}

We now indicate how to express the proof of the result given by Brown and Ozawa~\cite[Theorem 4.3.4]{bo}
on amenable actions and nuclearity of the reduced crossed product in our language.
Assume for simplicity that $A$ is a unital $C^*$-algebra with identity $\unA$.

\begin{definition}\label{d_45}
We say that an action $\alpha$ of a discrete group $G$ on a unital  $C^*$-algebra $A$ is \emph{amenable}
if there exists a net $(T_i)_{i\in I}$ of
finitely supported functions $T_i : G \to \mathcal{Z}(A)^+$ such that
$\sum_{i\in I} T_i(t)^2 = \unA$ and
$$\left\|\sum_{s \in G}(T_i(s) - \alpha_t(T_i(t^{-1}s)))^* (T_i(s) - \alpha_t(T_i(t^{-1}s)))\right\| \to_{i\in I} 0,$$
for every $t\in G$.
\end{definition}

\begin{corollary}\label{co:amenablenuclear}
Assume that $G$ is a discrete group acting amenably on a unital, $C^*$-algebra $A$.
Then $A\rtimes_{\alpha,r} G$ is nuclear if and only if $A$ is nuclear.
\end{corollary}
\begin{proof}
The forward implication follows from the existence of conditional expectation from $A\rtimes_{\alpha,r}G$ onto $A$, and does not require amenability of the action (see Remark~\ref{rem:nuclear}).

Assume that $A$ is nuclear and the action of $G$ on $A$ is amenable.
We will show that if $(\Phi_j)_{j \in J }$ is an approximating net of unital completely positive  maps for $A$ and
$(T_i)_{i \in I}$ is a net as in Definition \ref{d_45} (where $T_i$ is supported on $\mathcal{F}_i$),
then the maps $F_{i,j}(s)$ ($i \in I, j \in J, s \in G$), given by
\[
    F_{i,j}(s) (a) = \sum_{p \in G} T_i(p) \alpha_p(\Phi_j(\alpha_p^{-1}(a)))\alpha_s (T_i(s^{-1}p)), \quad  a \in A,
\]
yield Herz-Schur multipliers providing the completely positive finite rank approximations for  $A\rtimes_{\alpha,r} G$.
Note that the summation above is in fact over the finite set $\mathcal{F}_i \cap s\mathcal{F}_i$.



We first show that $F_{i,j}$ is a completely positive Herz-Schur $(A,G,\alpha)$-multiplier.
For $s$, $t\in G$  and $a\in A$, we have
\[ \begin{split}
    \cl N(F_{i,j})(s,t)(a) &= \sum_{p \in G} 
\alpha_{s^{-1}}(T_i(p) \alpha_p(\Phi_j(\alpha_p^{-1}(\alpha_s(a))))\alpha_{st^{-1}} (T_i(ts^{-1}p))) \\
    &= \sum_{p \in G} 
\alpha_{s^{-1}}(T_i(p)) \alpha_{s^{-1}p}(\Phi_j(\alpha_{s^{-1}p}^{-1}(a)))\alpha_{t^{-1}} (T_i(ts^{-1}p))) \\
    &=\sum_{p \in G} 
\alpha_{s^{-1}}(T_i(sp)) \alpha_{p}(\Phi_j(\alpha_{p}^{-1}(a)))\alpha_{t^{-1}} (T_i(tp))).
\end{split} \]
As each $\Phi_j$ is a unital completely positive map on $A\subseteq B(H)$,
there exist Hilbert spaces $H_{p,j}$, $*$-representations $\pi_{p,j}:A\to \cl B(H_{p,j})$, and bounded operators $V_{p,j}\in \cl B(H, H_{p,j})$, such that
$$\alpha_p(\Phi_j(\alpha_p^{-1}(a)))=V_{p,j}^*\pi_{p,j}(a)V_{p,j}, \ \ \ p\in G, j\in J, a\in A.$$
Moreover,
\[
    \|\alpha_p\circ\Phi_j\circ\alpha_p^{-1}\|_{\rm cb}=\|\Phi_j\|_{\rm cb}=\|V_{p,j}\|^2=1.
\]
Let $\rho_j:=\oplus_{p\in G}\pi_{p,j}$ and let $V_{i,j}(s):H\to\oplus_{p\in G}H_{p,j}$ be the column operator $( V_{p,j}\alpha_{s^{-1}}(T_i(sp)))_{p\in G}$. 
We then get
\[
    \cl N(F_{i,j})(s,t)(a)=V_{i,j}(s)^*\rho_j(a)V_{i,j}(s)
\]
with
\[ \begin{split}
    \|V_{i,j}(s)\|^2 &= \left\|\sum_{p\in G}\alpha_{s^{-1}}(T_i(sp))^*V_{p,j}^*V_{p,j}\alpha_{s^{-1}}(T_i(sp))\right\| \\
    &\leq \left\|\sum_{p\in G}\alpha_{s^{-1}}(T_i(sp))^2\right\| = \left\|\sum_{p\in G} T_i(p)^2\right\| = 1.
\end{split} \]
By Theorems~\ref{re:Sphicpintermsofphi} and \ref{co:equivalencesofpositiveHSmultiplier}, $F_{i,j}$ is a completely postive Herz-Schur $(A,G,\alpha)$-multiplier.
Furthermore,
\[
    F_{i,j}(e)(a)=\sum_{p \in \mathcal{F}_i} T_i(p) \alpha_p(\Phi_j(\alpha_p^{-1}(a)))T_i(p)
\]
and, since each $\Phi_j$ is a unital completely positive map,
\[
    \|F_{i,j}(e)\|_{\rm cb} = \|F_{i,j}(e)(1_A)\| \leq \left\|\sum_{p\in G} T_i(p)^2\right\| = 1.
\]
Since $\Phi_j\in\cl F(A)$, we have that $F_{i,j}(s)\in\cl F(A)$ for all $s\in G$, $i\in I$ and $j\in J$. Finally,
\[ \begin{split}
    \|F_{i,j}(s)(a)-a\| &= \|\sum_{p \in \mathcal{F}_i \cap s\mathcal{F}_i} T_i(p) \alpha_p(\Phi_j(\alpha_p^{-1}(a)))\alpha_s (T_i(s^{-1}p))-a\| \\
    &\leq \|\sum_{p \in \mathcal{F}_i \cap s\mathcal{F}_i} T_i(p) \alpha_s (T_i(s^{-1}p))\|\|\alpha_p(\Phi_j(\alpha_p^{-1}(a))-\alpha_p^{-1}(a))\| \\
    &\quad + \|\sum_{p \in \mathcal{F}_i \cap s\mathcal{F}_i} T_i(p)\alpha_s (T_i(s^{-1}p))-1_A)\|\|a\| .
\end{split} \]
The latter converges to zero for any $s\in G$  by \cite[Lemma 4.3.2]{bo}.
\end{proof}

\noindent {\bf Acknowledgements. }
A.S.\ was partially supported by the National Science Centre (NCN) grant no.~2014/14/E/ST1/00525.
The present work started during visits of A.S. and L.T. to Queen's University Belfast in 2016, which they gratefully acknowledge.
A.McK. is grateful for the hospitality at the Institute of Mathematics of the Polish Academy of Sciences during a visit in March 2017.
I.G.T. acknowledges the hospitality of the Department of Mathematical Sciences of the Chalmers University of Technology and the
University of Gothenburg during his visit in April 2017.

\end{document}